\documentclass[twoside,12pt,reqno]{amsart}
\usepackage{amsmath,amscd,amssymb, epsfig}%,pstricks}%epsfig} %,psfrag,graphicx

\newcommand{\C}{\mathbb{C}}
\newcommand{\N}{\mathbb{N}}
\newcommand{\sll}{\mathfrak{sl}}

\newcommand{\qum}[1]{\widetilde{#1}}

\newcommand{\g}{\ensuremath{\mathfrak{g}}}

\newcommand{\Vbar}{\ensuremath{\overline{V}}}

\newcommand{\what}[1]{\widehat{#1}}

\def\M{{\mathcal{M}}}
\def\T{{\mathcal{T}}}
\newcommand{\Atangle}{\ensuremath{A}}
\newcommand{\A}{\ensuremath{\mathcal{A}}}
\newcommand{\Uh}{\ensuremath{U_{h}(\mathfrak{g}) } }

\newcommand{\slmn}{\sll(m|n)}
\newcommand{\osp}{\mathfrak{osp}(2|2n)}

\newtheorem{definition}{Definition}
\newtheorem{theorem}[definition]{Theorem}

\newtheorem{lemma}[definition]{Lemma}

\newcommand{\epsh}[2]
         {\begin{array}{c} \hspace{-1.0mm} 
        \raisebox{-4pt}{\epsfig{figure=#1.eps,height=#2}}
        \hspace{-1.0mm}
        \end{array}}

\begin{document}
\let\co=\comment \let\endco=\endcomment
\title[The Kontsevich integral and re-normalized link invariants]{The Kontsevich integral and re-normalized link invariants arising from Lie superalgebras}
\author{Nathan Geer}
\address{School of Mathematics\\
  Georgia Institute of Technology\\
  Atlanta, GA 30332-0160, USA\\
  and\\
 Max-Planck-Institut f\"ur Mathematik\\
Vivatsgasse 7\\
53111 Bonn, Germany} \email{geer@math.gatech.edu} 
  \date{\today}

% \primaryclass{57M27}
 
\begin{abstract}
We show that the coefficients of the re-normalized link invariants of \cite{GP2} are Vassiliev invariants which give rise to a canonical family of weight systems. 
\end{abstract}

\maketitle
\setcounter{tocdepth}{1}
% \tableofcontents

%%%%%%%%%%%%%%%%%%%%%%%%%%%%%%%%%%%%%%%%%%%%%%%%%%%%%%%%%%%%%%%%%%%%%%%%%%%%%%
%%%%%%%%%%%%%%%%%%%%%%%%%%%%%%%%%%%%%%%%%%%%%%%%%%%%%%%%%%%%%%%%%%%%%%%%%%%%%%
\section*{Introduction}
Given a sequence of finite dimensional representations $\Vbar=\{V_1,V_2,...\}$ of a finite dimensional semisimple Lie algebra $\g$ one can construct the following two invariants of links (with ordered components): 
\begin{enumerate}
	\item the Reshetikhin-Turaev $\C[[h]]$-valued quantum group invariant $Q_{\g,\Vbar}$ which arises from $\Vbar$ and the Drinfeld-Jimbo quantization associated to $\g$ (see \cite{RT}),\label{I:PaperQG}
	\item $W_{\g,\Vbar} \circ Z$ where $W_{\g,\Vbar}$ is a weight system, constructed by Bar-Natan in \cite{BN2}, and where $Z$ is the Kontsevich integral \cite{Kont}. \label{I:PaperWS}
\end{enumerate}
Here a link or chord diagram (with ordered components) is colored by assigning the $i$th representation $V_i$ to its $i$th component.  The above constructions are essentially the same in the following sense.   Lin \cite{Lin} showed that the $m$th coefficient of $Q_{\g,\Vbar}$ is a Vassiliev invariant of type $m$.  Moreover, there is a weight system corresponding to  $Q_{\g,\Vbar}$ which can be shown to be equal to $W_{\g,\Vbar}$.  Conversely, Le and Murakami \cite{LM} show that $W_{\g,\Vbar}$ is canonical, i.e. the invariant $W_{\g,\Vbar}\circ Z$ is equal (up to a change of variable and normalization) to  $Q_{\g,\Vbar}$.    

In \cite{G04B} it is shown that there are analogous results for Lie superalgebras of type A-G.   The theory of Lie superalgebras has properties which create new challenges and interesting consequences.  First, the proof of Le and Murakami uses results, due to Drinfeld, whose proofs are based on properties of Lie algebras which fail for Lie superalgebras.  In \cite{G04B} the author over comes this difficulty by giving a proof which uses new quantum group results.  Second, for many sequence of representations $\Vbar$ of a Lie superalgebra $\g$ the quantum invariant $Q_{\g,\Vbar}$ is zero (see \cite{GP2} and the referernce within).  However, in \cite{GP2} it is shown that the usual quantum invariants associated to Lie superalgebras of type I can be re-normalized by modified quantum dimensions which lead to non-trivial invariants of links.   These invariants contain multivariable invariants which specialize to the multivariable Conway potential function.  In this paper we will show that the coefficients of these re-normalized invariants are Vassiliev invariants which give rise to canonical weight systems.  We will discuss how these results suggest that there is a natural choice for the modified quantum dimensions for quantized Lie superalgebras of type I.

\subsection*{Acknowledgments}    
I would like to thank Bertrand Patureau-Mirand for helpful discussions.   This work has been partially supported by the NSF grant DMS-0706725.   %This is the NSF Grant given for 2007-2010.

%%%%%%%%%%%%%%%%%%%%%%%%%%%%%%%%%%%%%%%%%
%%%%%%%%%%%%%%%%%%%%%%%%%%%%%%%%%%%%%%%%%%
%%%%%%%%%%%%%%%%%%%%%%%%%%%%%%%%%%%%%%%%%

\section{Quantum $\g$ and its associated ribbon function}
Throughout all links and tangles will have components which are ordered, framed and oriented.  Let $\g$ be a Lie superalgebras of type I and let $h$ be an indeterminate.  Let $\Uh$ be the braided quantized Lie superalgebra over $\C[[h]]$ associated to $\g$ (see \cite{GP2} and references within).    We say a $\Uh$-module $W$ is topologically free of finite rank if it is isomorphic as a $\C[[h]]$-module to $V[[h]]$, where $V$ is a
finite-dimensional $\g$-module. 
The set of isomorphism classes of irreducible finite-dimensional
$\g$-modules are in one to one correspondence with the set of dominant
weights.   Each highest weight
$\g$-module $V$ can be deformed to a highest weight topologically free
$\Uh$-module $\qum{V}$ which is equal to $V[[h]]$.
 
 Let $\M$ the category of
topologically free of finite rank $\Uh$-modules.  A standard argument
shows that $\M$ is a ribbon category (for details see \cite{G04B}). 
Let $\T=Rib_\M$ be the ribbon category of framed oriented tangles colored by elements of $\M$ in
the sense of Turaev (see \cite{Tu}).   Let $F$ be the usual ribbon
functor from $\T$ to $\M$ (see \cite{Tu}).

\section{The Kontsevich integral and (1,1)-tangle invariants arising from $\g$.}
In this section we will recall that the quantum invariants arising from representations of  $\g$ are equal to the composition of the Kontsevich integral and certain weight systems. 

We recall the notions of Vassiliev invariants, for more details see \cite{BN2, Kas, G04B}.  To make a consistent theory of Vassiliev invariants of framed links we restrict to framed links with even framings.  

By a \emph{singular link} we mean a link with a finite number of self-intersections, each having distinct tangents.  Any numerical link invariant $f$ can be inductively extended to an invariant of singular link according to the rule
$$
   % \rotatebox{90}{\scalebox{0.4}{\includegraphics{cent-elem.pdf}}}
   \put(-10,-2){{\Large $f$}}  \epsh{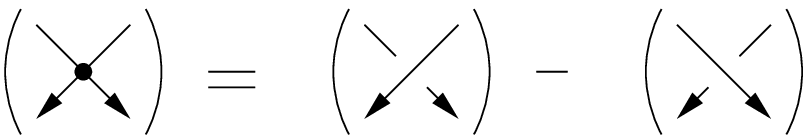}{7ex} \put(-60,-2){{\Large $f$}} \put(-149,-2){{\Large $f$}}
$$

A \emph{Vassiliev invariant} \cite{Vas} of type $m$ is a framed link invariant whose extension vanishes on any framed singular link with more than $m$ double points.  
Similarly, a Vassiliev $(1,1)$-tangle invariant of type $m$ is a framed $(1,1)$-tangle invariant whose 
extension vanishes on any framed singular $(1,1)$-tangle with more than $m$ double points.

Let  $\Vbar=\{V_1,V_2,...\}$  be a sequence of simple finite dimensional representations of $\g$.   Let $\what{Q}_{\g,\Vbar}$ be the Reshetikhin-Turaev type $\C[[h]]$-valued quantum group invariant of $(1,1)$-tangles associated to $\g$ and $\Vbar$.  Let us briefly describe how this invariant is defined, for more details see \cite{G04B}.  Let $T$ be a  $(1,1)$-tangle and let us assume that the open component is labeled with $1$ and is oriented down.  Color the $i$th component of $T$ with $\qum{V}_i$ then $F(T)$ is an endormorphism of $\qum{V}_1$.   Since $V_1$ is simple it follows that this endomorphim is a scalar times the identity.  Then $\what{Q}_{\g,\Vbar}(T)$ is defined to be this scalar.

 The pair $(\g,\Vbar)$ also defines a weight system as follows.   Let $T$ be a tangle.
A chord diagram on $T$ of degree $m$ is the tangle $T$ with a distinguished set of $m$ unordered pairs of points of $T \backslash \partial T$, considered up to homeomorphisms preserving each connected component and the orientation.   Let $\Atangle(T)$ be the vector space with basis given by all chord diagrams on $T$ modulo the four term relation.

We will now describe the category of chord diagrams on tangles, which we denote as $\A$.  The objects of $\A$ are the empty set and finite sequences of pairs $(\epsilon,i)$ where $\epsilon=\pm$ and $i\in\N$.  The morphisms of $\A$ are elements of $\Atangle(T)$ for some tangle $T$.  Here each pair  $(\epsilon,i)$ is associated to a point in the boundary of the tangle, where $\epsilon$ and $i$ correspond to the orientation and the labeling of the component, respectively.   As in \cite{G04B} the category $\A$ is a strict infinitesimal symmetric category with duality. 
  
Let $U(\g)$-mod be the category of finite dimensional $\g$-modules.  As shown in \cite{G04B}, 
$U(\g)$-mod is a strict inÞnitesimal symmetric category with duality (here we fix the standard non-degenerate supersymmetric invariant even 2-tensor).  The following lemma is well known (see \cite{Kas}).
\begin{lemma}
 There exists  a unique functor
\begin{equation}
\label{E:FuncG}
G_{\g,\Vbar}:\A \rightarrow U(\g)\text{-}Mod
\end{equation}
preserving the tensor product, symmetry, infinitesimal braiding and the duality such that $G_{\g,\Vbar}((+,i))=V_i$.
\end{lemma}
Let $\Atangle(1,1)$ be the vector space of chord diagrams on $(1,1)$-tangles modulo the four term relation (here we will assume the open component is labeled with $1$).  Let $D\in \Atangle(1,1)$ then by construction we have $G_{\g,\Vbar}(D)$ is an endomorphism of $V_1$ and thus a complex number times the identity.  Define $\what{W}_{\g,\Vbar}(D)$ to be this complex number.   

\begin{theorem}\label{T:QWZ11} 
We have
\begin{enumerate}
  \item the $m$\textsuperscript{th} coefficient of $\what{Q}_{\g, \Vbar}$ is a Vassiliev invariant of type $m$,\label{TI:Vas}
  \item  $\what{Q}_{\g, \Vbar}=\what{W}_{\g,\Vbar}\circ Z$,\label{TI:WZQ}
  \item \label{TI:Vas3} the weight systems corresponding to  $\what{Q}_{\g, \Vbar}$ are equal to the family   $\what{W}_{\g,\Vbar}$.
\end{enumerate}
\end{theorem}
\begin{proof}
The proof of this theorem is almost identical to the proof of Theorems 5.2 and 5.5 in \cite{G04B}.  The only difference is that here links are allowed to be colored with more than one module.  One can check that the proofs of \cite{G04B} can easily be adapted to compensate for this difference and so we will not repeat the proof here.
\end{proof}

%%%%%%%%%%%%%%%%%%%%%%%%%%%%%%%%%%%
%%%%%%%%%%%%%%%%%%%%%%%%%%%%%%%%%%%
\section{Re-normalized link invariants}

Let $\g$ be a Lie superalgebra of type I, i.e.
$\g$ is equal to $\slmn$ or $\osp$.  Here we assume that $m\neq n$.
Let $r$ be equal to $m+n-1$ if $\g=\slmn$ and $n+1$ if $\g=\osp$.  
%Every irreducible finite-dimensional $\g$-module has a highest weight $\lambda\in\h^*$ (where $\h$ is the Cartan sub-superalgebra).
The set of isomorphism classes of irreducible finite-dimensional
$\g$-modules are parameterized by $\N^{r-1}\times\C$ and are
divided into two classes: typical and atypical.   For $a\in\N^{r-1}\times\C$ we denote the corresponding $\g$-module by $V(a)$.   We say $\qum{V}$ is a typical $\Uh$-module if $V$ is a
typical $\g$-module.

If $V$ is a typical $\g$-module then the super-dimension of $V$ is zero and so it follows that the quantum dimension of $\qum{V}$ is zero.  Thus, it follows that if $L$ is a link colored with elements of $\M$ such that at least one of these colors is a typical $\Uh$-module then $F(L)=0$.  

We will now explain how to use $F$ to construct a non-zero link invariant.  
Let $V$ be a typical $\Uh$-module.  If $T_{V}$ is a framed $(1,1)$-tangle colored by
 $\Uh$-modules such that the open string is colored by  a typical module $V$, then
$F(T_{V})=x Id_{V}$ for some $x$ in $\C[[h]]$.  Let us set $<F(T_{V})>=x$.
 In \cite{GP2}  Geer and Patureau define a map $d$ from the set of typical representations of $\Uh$ to the ring
$\C[[h]][h^{-1}]$.  Let us rescale $d$ by $h^{|\Delta_{\bar1}^+|}$ where $\Delta_{\bar1}^+$ is the set of odd positive roots of $\g$.  We will still denote this rescaled function by $d$, then $d$ takes values in $\C[[h]]$.  The assignment $T_V\mapsto d(V)<F(T_{V})>$  induces a well defined invariant of framed links.  In particular, in \cite{GP2} the following theorem is proved.
\begin{theorem} \label{thF'}
 Let $L$ be a framed link colored by $\Uh$-modules such that at least one color is typical.  Cut $L$ to obtain a $(1,1)$-tangle $T_V$ whose open string is colored by a typical module $V$.   Then the map given by $F': L\mapsto d(V)<F(T_{V})>$ is independent of the cut, i.e. $F'$ is a well defined framed colored link invariant.
  \end{theorem}

Note that any scalar of $d$ also defines a link invariant.  In \cite{GPT} it is shown that $d$ is the unique function up to a constant such that the assignment in Theorem \ref{thF'} gives a well defined invariant.  In the next section we explain how the Kontsevich integral suggests that there is a natural choice for the scaling of $d$.

\section{The invariant $Q'_{\g,\Vbar}$ and the Kontsevich integral}
Let $\Vbar=\{V_1,V_2,...\}$ be a sequence of finite dimensional $\g$-modules, such that $V_1$ is typical.  Let $L$ be a framed oriented link with ordered components.  Define $Q'_{\g,\Vbar}$ to be the $\C[[h]]$-valued invariant of $L$ given by $F'(L_{\Vbar})$ where $L_{\Vbar}$ is the link $L$ whose $i$th component is colored by $\qum{V}_i$. 

Let us use the following notation:
$$Q'_{\g,\Vbar}=\sum_{m=0}^{\infty}Q_m'h^m, \;\;\; \what{Q}_{\g, \Vbar}= \sum_{m=0}^{\infty}\what{Q}_mh^m, \;\;\; d=\sum_{m=0}^{\infty}d_mh^m.$$
Now $Q'_{\g,\Vbar}(L)=d(\qum{V}_i)\what{Q}_{\g, \Vbar}(T_i)$ where $T_i$ is a $(1,1)$-tangle coming from cutting the $i$th component of $L$.

\begin{lemma}
The coefficient $Q'_m$ is a Vassiliev invariant of type $m$ whose weight system is given by the assignment 
$$
D \mapsto d_0(\qum{V}_i)\what{W}_{\g,\Vbar}(D_i)
$$
where $D_i$ is an element of $\Atangle(1,1)$ coming from cutting the $i$th component of $D$.  We will denote this weight system by $W'_{\g,\Vbar}$.
\end{lemma}
\begin{proof}
First, since $Q'_m(L)=\sum_{j=0}^{m}d_j(\qum{V}_i)\what{Q}_{m-j}(T_i)$ we have that  Theorem \ref{T:QWZ11} implies $Q'_m$ is a Vassiliev invariant of type $m$.  By definition the weight system coming from $Q'_m$ is given by $D\mapsto Q'_m(K_D)$ where $K_D$ is any framed singular link with $m$ double points whose underlying diagram is $D$.  Then
$$Q_m'(K_D)=\sum_{j=0}^{m}d_j(\qum{V}_i)\what{Q}_{m-j}(T_i)=d_0(\qum{V}_i)\what{Q}_m(T_i)$$
since $T_i$ is a singular $(1,1)$-tangle with $m$ double points and $\what{Q}_{m-j}$ is a Vassiliev invariant of type $m-j$.  The lemma is completed by Theorem \ref{T:QWZ11} \eqref{TI:Vas3} which states that for all $D \in \Atangle(1,1)$ with $m$ chords we have $\what{W}_{\g,\Vbar}(D)=\what{Q}_m(L_D)$ where $L_D$ is any framed singular $(1,1)$-tangle whose underlying diagram is $D$.
\end{proof}

\begin{theorem}\label{T:QConical}
The invariant $W'_{\g,\Vbar}$ is canonical, i.e. up to normalization $Q'_{\g,\Vbar}$ is equal to $W'_{\g,\Vbar}\circ Z$.
\end{theorem}
\begin{proof}
Let $L$ be a link whose $i$th component is colored with $\qum{V}_i=V_i[[h]]$.  Then  the assignment 
\begin{equation}\label{E:H}
L\mapsto H(L), \text{ given by } H(L)=(W'_{\g,\Vbar}\circ Z)(L)
\end{equation}
is a well defined invariant of the colored link $L$.  Here in the right side of the equality in Equation \eqref{E:H} we ignore the coloring of $L$.  In the rest of the proof in similar situations we will ignore the coloring of links and  
tangles.

Recall that $V_1$ is a typical module.  Suppose $L$ is equal to the closure of a $(1,1)$-tangle $T_{\qum{V}_1}$ whose open string is the first component of $L$.  Then
$$H(T_{\qum{V}_1}):=(G_{\g,\Vbar}\circ Z)(T_{\qum{V}_1})$$
is an endomorphism of $\qum{V}_1$ which satisfies $H(L)=H(U_{\qum{V}_1})<H(T_{\qum{V}_1})>$ where $U_{\qum{V}_1}$ is the unknot colored with $\qum{V}_1$.  Then we have
\begin{align*}
(W'_{\g,\Vbar}\circ Z)(L)&=H(U_{\qum{V}_1})<H(T_{\qum{V}_1})>\\
&=(W'_{\g,\Vbar}\circ Z)(U_{\qum{V}_1})<(G_{\g,\Vbar}\circ Z)(T_{\qum{V}_1})>\\
&=(W'_{\g,\Vbar}\circ Z)(U_{\qum{V}_1})(\what{W}_{\g,\Vbar}\circ Z)(T_{\qum{V}_1})\\
&=(W'_{\g,\Vbar}\circ Z)(U_{\qum{V}_1})\what{Q}_{\g, \Vbar}(T_{\qum{V}_1})
\end{align*}
where the last equality follows from Theorem \ref{T:QWZ11} \eqref{TI:WZQ}.  Finally,  the uniqueness of the invariant in Theorem \ref{thF'} implies that $d(\qum{V}_1)$ must be equal to a multiple of $(W'_{\g,\Vbar}\circ Z)(U_{\qum{V}_1})$.  Thus, $W'_{\g,\Vbar}$ is canonical.
\end{proof}
The proof of the theorem suggests that $(W'_{\g,\Vbar}\circ Z)(U_{\qum{V}_1})$ is a natural choice of the normalization of $Q'_{\g,\Vbar}$.   %One of the main technical requirements in constructing $F'$ or $Q'_{\g,\Vbar}$ is the existence of an ambidextrous object (see \cite{GPT}).  The proof of Theorem \ref{T:QConical} also suggest that ambidextrous objects in other quantized category maybe constructed from the underlying non-quantized category.  This could be useful because many non-quantized categories have essentially trivial braidings making the definition of ambidextrous easier to check.   However, before attempting such generalizations one must  complete a through study of the notion of ambidextrous.  Such a study could be very involved but potentially fruitful.  

\linespread{1}

\vfill

\end{document}